\documentclass[11pt,a4paper]{article}
\usepackage[utf8]{inputenc}
\usepackage[english]{babel}
\usepackage{amsmath,amssymb,amsthm, esint}
\usepackage[left=2cm,right=2cm,top=2cm,bottom=2cm]{geometry}

\usepackage{indentfirst} 
\usepackage[affil-it]{authblk}

\usepackage{cite}

\theoremstyle{plain}
\newtheorem{lemma}{Lemma}[]
\newtheorem{theorem}[lemma]{Theorem}

\newtheoremstyle{dotless}{}{}{\itshape}{}{\bfseries}{}{ }{}
\theoremstyle{dotless}
\newtheorem*{void*}{}
\theoremstyle{definition}

\theoremstyle{remark}

\newcommand{\R}{\mathbb{R}}

\author{Antonio Córdoba, Jesús Ocáriz \\
Universidad Autónoma de Madrid (UAM), Instituto de Ciencias Matemáticas (ICMAT)} 

\title{A note on generalized laplacians and minimal surfaces}

\begin{document}

\maketitle
\begin{abstract}
In these notes we give an interdisciplinary result which links the geometric concept of minimal surfaces with generalized harmonic functions.
\end{abstract} 

\section*{Introduction}

Let $f$ be a locally integrable function defined on an open domain of the euclidean space $\R^n$, its generalized laplacian $\widehat{\Delta}f$ is given by the following limit (provided it exists):
$$
\widehat{\Delta}f(x):=\lim_{r\to 0^+} \frac{2(n+2)}{r^2}\fint_{B^{(n)}_r(x)} \left( f(y)-f(x) \right) \mathrm{d}y.
$$
Here the symbol $\fint_{B^n_r(x)}$ denotes the average on the ball $B^{(n)}_r(x)\subset \R^n$ centered at $x$ and with radius $r$; $\mathrm{d}y = \mathrm{d}\mu_n(y)$ is the Lebegue measure in $\R^n$ and $\mu_n(A)=|A|=\text{vol}(A)$ is the measure of the measurable set $A$.

\vspace{0.5cm}

It is well known (\cite{rado2013subharmonic}, \cite{riesz1926fonctions}) that a continuous generalized harmonic function (i.e. $\widehat{\Delta}f(x)=0$ for every $x$) must be smooth and, therefore, harmonic in the ordinary sense: $\Delta f=0$.

\vspace{0.5cm}

There are, however, discontinuous functions which are harmonic in the generalized sense. An example is given by:
$$
f(x', x_n)=\begin{cases} 
      \alpha_+ & \text{if } x_n > 0 \\
      \frac{1}{2}(\alpha_+ + \alpha_-), & \text{if } x_n = 0 \\
      \alpha_- & \text{if } x_n < 0
   \end{cases}
$$
which is clearly discontinuous when $\alpha_+\neq \alpha_-$ and, nevertheless, satisfies $\widehat{\Delta}f(x)=0$ everywhere.

\vspace{0.5cm}

Suppose that $S$ is a $C^1$-hypersurface separating the domain $\Omega$ in two non empty components: $\Omega = \Omega^+\cup \Omega^- \cup S$, 
$$
S=\partial \Omega^+ \cap \Omega = \partial \Omega^- \cap \Omega.
$$

\vspace{0.5cm}

Let the function $f_S$ be equal to $\alpha_+$ inside $\Omega^+$, $\alpha_-$ inside $\Omega^-$ and $\frac{1}{2}(\alpha_+ + \alpha_-)$ in $S$. 

\vspace{0.5cm}

The main purpose of this note is to give a proof of the following:
\begin{theorem}
The function $f_S$ (with $\alpha_+\neq \alpha_-$) is generalized harmonic if and only if $S$ is minimal.
\end{theorem}

\vspace{0.5cm}

In the proof we will make use of the modern notion of viscosity solution of uniformly elliptic equations. Namely, we will show that if $\widehat{\Delta} f_S=0$, then, locally, $S$ will be given as the graph of a viscosity solution of the minimal surface equation and, therefore, it has to be smooth. We shall use also several well-known properties of minimal surfaces and elliptic equations for which \cite{caffarelli1993elementary}, \cite{giusti1984minimal} and \cite{roberts1995fully} are appropriate references.

\section*{A basic calculus lemma}

Let $S$ be a smooth ($C^4$) hypersurface in $\R^n$ with unit normal vector field $\nu(x)$. Given $x\in S$ and $r>0$, small enough, the ball $B^{(n)}_r(x)$ is separated by $S$ in two connected components, $S^+_r(x)$, $S^-_r(x)$, where $S^+_r(x)$ (respectively $S^-_r(x)$) consists of the points inside $B^{(n)}_r(x)$ which are placed above (respect. below) $S$ in the given normal direction.

\vspace{0.5cm}

Then we have:
\begin{lemma} As $r\to 0^+$,
$$
\text{vol}(S^+_r(x))-\text{vol}(S^-_r(x)) = - c_n H_S(x) \cdot r^{n+1}+ O(r^{n+3}).
$$
Here $H_S(x)$ denotes the mean curvature of $S$ at the point $x$ and $c_n>0$ is a universal constant that only depends on the dimension $n$.
\end{lemma}

\begin{proof}
Without loss of generality we can assume that $x=0$ is the origin of a coordinate system such that the tangent space of $S$ at $x$ is horizontal, i.e., the normal vector is $\nu(0)=(0,\cdots,0,1)$. Hence, near $x=0$, $S$ is the graph of a smooth ($C^4)$) function $x_n=\varphi(x_1,\cdots,x_{n-1})$ satisfying:
\begin{enumerate}
\item $\varphi(0)=0$.
\item $\nabla \varphi(0)=0$.
\end{enumerate}

\vspace{0.5cm}
Then, inside the cylinder $B^{(n-1)}_r(0)\times \R$, $r$ small enough, we have the inclusion
$$
S\cap \left(B^{(n-1)}_r(0)\times \R\right)  \subset \{x\in \R^n: |x_n|\leq c_1 r^2\}
$$
for a positive constant $c_1$ depending upon the size of the second derivatives of $\varphi$. 

\vspace{0.5cm}

An elementary computation shows that the vertical projection of $S\cap B^{(n)}_r(0)$ onto $B^{(n-1)}_r(0)$ must contain the ball
$$
B^{(n-1)}_{r-c_2r^3}(0)
$$
for a fixed constant $c_2$.

\vspace{0.5cm}

Let
$$
D^+_r=B^{(n)}_r(0)\cap \Omega^+ \cap \{|x_n|\leq c_1 r^2\},
$$
$$
D^-_r=B^{(n)}_r(0)\cap \Omega^- \cap \{|x_n|\leq c_1 r^2\}.
$$

\vspace{0.5cm}

Then, since $S$ is contained in the strip, by symmetry we get the first equality
\begin{align*}
\text{vol}(S^+_r(0))-\text{vol}(S^-_r(0)) &=\text{vol}(D^+_r)-\text{vol}(D^-_r)\\
 &=\int_{B^{(n-1)}_{r-c_2r^3}(0)} \left(c_1r^2-\varphi(x)\right) \mathrm{d}x - \int_{B^{(n-1)}_{r-c_2r^3}(0)} \left(c_1r^2+\varphi(x)\right) \mathrm{d}x + I
\end{align*}
where $I$ is the correction of restricting the domain of integration.

\vspace{0.5cm}

A direct computation of the term $I$ shows that
$$
|I| \lesssim r^2\cdot (r^{n-1}-(r-c_2r^3)^{n-1}) \lesssim r^{n+3}= O(r^{n+3}).
$$

\vspace{0.5cm}

Hence
$$
\text{vol}(S^+_r(0))-\text{vol}(S^-_r(0)) = -2 \int_{B^{(n-1)}_{r-c_2r^3}(0)} \varphi(x) \mathrm{d}x + O(r^{n+3}).
$$

\vspace{0.5cm}

With another computation with respect to the volume of the cylinders, we obtain that
$$
\left|\int_{B^{(n-1)}_{r}(0)}\varphi(x) \mathrm{d}x - \int_{B^{(n-1)}_{r-c_2r^3}(0)} \varphi(x) \mathrm{d}x  \right| \lesssim r^{n+3} = O(r^{n+3}).
$$ 

\vspace{0.5cm}

Then Taylor's expansion yields
$$
\fint_{B^{(n-1)}_r(0)} \varphi(y) \mathrm{d}y = \frac{1}{2(n+1)} \Delta \varphi (0) \cdot r^2 + O(r^4).
$$

\vspace{0.5cm}

This allows us to finish the proof of the lemma, because we know that
$$
H_S (x',\varphi(x')) = \frac{1}{n-1} \text{div}\left( \frac{\nabla \varphi}{\sqrt{1+|\nabla \varphi|^2}}  \right)(x')
$$

\vspace{0.5cm}

And since $\nabla \varphi(0)=0$, we have
$$
H_S(0)= \frac{1}{n-1}\Delta \varphi (0).
$$
\end{proof}

\section*{Proof of Theorem 1}

First, without loss of generality, one can assume that $\alpha_+= +1$ and $\alpha_-=-1$. Next, let us observe that one of the two implications of the theorem follows immediately: namely if $S$ is minimal and $C^1$ then, by the classical theorem of de Giorgi-Nash, $S$ has to be smooth and we can apply Lemma 2 to observe that at any point $x\in S$ we have:
\begin{align*}
\widehat{\Delta} f(x) &= \lim_{r\to 0^+} \frac{2(n+2)}{r^2} \fint_{B_r^{(n)}(x)}(f(y)-0) \mathrm{d}y \\
&=\lim_{r\to 0^+} \frac{2(n+2)}{|B_r^{(n)}(x)|r^2} \left[\text{vol}(S^+_r(x))-\text{vol}(S^-_r(x)) \right] \\
&=0,
\end{align*}
because of the minimality condition $H_S(x)=0$.

\vspace{0.5cm}

Note, that a similar argument with Lemma 2 also works to prove that $f_S$ being generalized harmonic implies that $S$ is minimal. Therefore, to finish the proof we just need to prove the regularity ($C^4$) of $S$.

\vspace{0.5cm}

To continue the proof let us recall now that a real continuous function $F$ defined on \\ $\Omega\times \R\times \R^n \times M^n$, where $M^n$ denotes the vector space of $n$ x $n$ symmetric matrices, yields an elliptic equation $F(x,u,Du, D^2u)=0$ if 
$$
F(x,u,\eta, \delta)\leq F(x,u,\eta, \delta + \sigma)
$$ 
for all $(x,u,\eta, \delta)\in \Omega\times \R\times \R^n \times M^n$ and $\sigma\in M^n$ non-negative.

\vspace{0.5cm}

The elliptic equation is called uniformly elliptic if there exist positive constant $\lambda, \Lambda$ satisfying the estimate:
$$
0<\lambda \Vert \sigma \Vert < F(x,u,\eta, \delta+\sigma) - F(x,u,\eta, \delta) \leq \Lambda \Vert \sigma \Vert
$$
where $\Vert \sigma \Vert$ denotes the $(L^2,L^2)$-norm (i.e. $\Vert \sigma \Vert= \sup_{\Vert x\Vert =1} \Vert \sigma x \Vert = $ maximum of the eigenvalues of $\sigma$).

\vspace{0.5cm}

\underline{Definition}: A continuous function $u$ is called a viscosity subsolution (respectively supersolution) of $F(x,u,Du, D^2u)=0$ if, for any quadratic polynomial $\Psi\in C^2(\Omega)$ and local maximum (respectively local minimum) $x_0$ of $u-\Psi$ we have 
$$
F(x_0,u(x_0),D\Psi(x_0), D^2 \Psi(x_0))\geq 0 \hspace{1cm} (\text{respect. } \leq 0).
$$ 
Finally $u$ is a viscosity solution if it is both a viscosity subsolution and supersolution.

\vspace{0.5cm}

Reference \cite{roberts1995fully} contains the result about regularity (Corollary 5.7) and uniqueness (Corollary 5.4) of viscosity solutions of uniformly elliptic equations, which we shall invoke to conclude the proof.

\vspace{0.5cm}

More precisely, under the hypothesis that $\widehat{\Delta} f\equiv 0$, let $P$ be a paraboloid tangent from below to our hypersurface $S=\{x_n=\varphi(x_1,\cdots,x_{n-1})$ at a point $x=(x_1, \cdots, x_{n-1}, x_n)\in S \cap P$.

\vspace{0.5cm}

Let $f_P$ be its corresponding function defined in the introduction, we have the inequality
$$
\fint_{B^{(n)}_r(x)} f_P \geq \fint_{B^{(n)}_r(x)} f_S
$$

\vspace{0.5cm}

On the other hand, lemma 2 applied to the hypersurface $P$ yields
$$
\fint_{B^{(n)}_r(x)} f_P = -c_n H_P(x) \cdot r^{n+1} + O(r^{n+2})
$$
which together with the hypothesis
$$
\lim_{r\to 0^+} \frac{1}{r^2}\fint_{B^{(n)}_r(x)} f_S = 0
$$
implies that $H_P(x)\leq 0$.

\vspace{0.5cm}

Similarly if $P$ is now a paraboloid tangent to $S$ from above at the point $x$, then we must have $H_P(x)\geq 0$. Therefore $\varphi$ is a viscosity solution of the equation 
$$
\text{div}\left( \frac{\nabla \varphi}{\sqrt{1+|\nabla \varphi|^2}}\right)=0
$$
whose uniform ellipticity is ensured by the hypothesis that $S$ is of class $C^1$.

\vspace{0.5cm}

The regularity theory of such solutions (\cite{roberts1995fully}) allows us to conclude the smoothness of $\varphi$ (and $S$) and, therefore, its minimality.

\bibliography{referencesMinSurf}
\bibliographystyle{ieeetr}

\end{document}